\documentclass{amsart}
\usepackage{amssymb, latexsym}

\newcommand{\beqa}{\begin{eqnarray*}}
\newcommand{\eeqa}{\end{eqnarray*}}
\newcommand{\beqn}{\begin{eqnarray}}
\newcommand{\eeqn}{\end{eqnarray}}

\newcommand{\N}{\mathbb N}

\newcommand{\mcH}{\mathcal H}

\newcommand{\mcR}{\mathcal R}

\newcommand{\mcM}{\mathcal M}

\newcommand{\la}{\lambda}

\newcommand{\m}{\mu}

\newcounter{cnt1}
\newcounter{cnt2}
\newcounter{cnt3}
\newcommand{\blr}{\begin{list}{$($\roman{cnt1}$)$}
 {\usecounter{cnt1} \setlength{\topsep}{0pt}
 \setlength{\itemsep}{0pt}}}
\newcommand{\bla}{\begin{list}{$($\alph{cnt2}$)$}
 {\usecounter{cnt2} \setlength{\topsep}{0pt}
 \setlength{\itemsep}{0pt}}}
\newcommand{\bln}{\begin{list}{$($\arabic{cnt3}$)$}
 {\usecounter{cnt3} \setlength{\topsep}{0pt}
 \setlength{\itemsep}{0pt}}}
\newcommand{\el}{\end{list}}
\newtheorem{thm}{Theorem}[section]

\newtheorem{cor}[thm]{Corollary}
\newtheorem{ex}[thm]{Example}

\newtheorem{Def}[thm]{Definition}
\newtheorem{prop}[thm]{Proposition}
\newtheorem{rem}[thm]{Remark}
\newcommand{\Rem}{\begin{rem} \rm}
\newcommand{\bdfn}{\begin{Def} \rm}
\newcommand{\edfn}{\end{Def}}

\newcommand{\ba}{\begin{array}}
\newcommand{\ea}{\end{array}}

\begin{document}
\begin{center}\large{{\bf Kluv\'{a}nek-Lewis-Henstock integral in a Banach space}} \\

 
  Hemanta Kalita$^1$ and Bipan Hazarika$^{2, \ast}$\\
  
 $^1$Department of Mathematics, Gauhati University, Guwahati 781014, Assam, India\\
Email:  
hemanta30kalita@gmail.com; hk\_gu@gauhati.ac.in\\
$^{2}$Department of Mathematics, Gauhati University, Guwahati 781014, Assam, India\\
Email:  
bh\_rgu@yahoo.co.in; bh\_gu@gauhati.ac.in
\end{center}
\title{}
\author{}
\thanks{$^{\ast}$The corresponding author}
\thanks{\today} 

\begin{abstract}  We investigate some properties and convergence theorem  of  Kluv\'{a}nek-Lewis-Henstock $\m-$integrability for $\m-$measurable functions that we introduced in \cite{ABH}. We give a $\m-$a.e. convergence version of Dominated (resp. Bounded) Convergence Theorem for $\m.$ We introduce  Kluv\'{a}nek-Lewis-Henstock integrable of scalar-valued functions with respect to a set valued measure in a Banach space. Finally we introduce $(KL)-$type Dominated Convergence Theorem for the set-valued  Kluv\'{a}nek-Lewis-Henstock integral.
{\footnotesize {\bf{AMS subject classification \textrm{(2020)}:}}} Primary 28B05,
46G10; Secondary 46B03, 46B25, 46B40.
\end{abstract}
\maketitle

\maketitle

\pagestyle{myheadings}
\markboth{\rightline {\scriptsize   Kalita, Hazarika }}
        {\leftline{\scriptsize  Kluv\'{a}nek-Lewis-Henstock integral \dots }}

\maketitle
\section{Introduction}
In 1955, Bartle, Dunford, and Schwartz  developed a theory of integration for scalar functions with respect to a $\sigma-$additive Banach-space-valued vector measure $\mu$ defined on a $\sigma-$algebra of sets and used it to give an integral representation for weakly compact operators $u : C(S) \to  X,$ where $S$ is a compact Hausdorff space and $X$ is a Banach space (see \cite{BDS}). About fifteen years later,  Lewis \cite{DR,DR1} studied a Pettis type weak integral of scalar functions with respect to a $\sigma$-additive vector measure $\mu$ having range in an locally convex Hausdorff  $X.$  This type of definition has also been considered by Kluv\'{a}nek in \cite{KK1}. In
honor of these mathematicians in the literature  the integral introduced in \cite{DR, DR1} as well as
its variants is called Kluv\'{a}nek-Lewis integral or briefly, the $(KL)-$integral. \cite[Theorem 2.4]{DR1}  asserts that the $(KL)-$integral is the same as the $(BDS)-$integral (see \cite{KK1,KK} for detailed about $(KL)-$integral). In 1960s, both Henstock and Kurzweil independently developed the so called
Henstock-Kurzweil (HK) integral, an integral able to integrate a much larger class of functions than the Riemann integral and the Lebesgue integral, compatible with the Fundamental Theorem of Calculus. The Henstock-Kurzweil integral provides a tool for integrating highly oscillatory functions which occur in nonlinear analysis and in quantum theory. It is also easy
to understand because its definition requires no measure theory. Moreover, all Bochner integrable (in real-valued case Lebesgue integrable) functions are Henstock-Kurzweil (shortly HK) integrable, but not conversely. The Henstock-Kurzweil (in short, $(HK)$)-integral  for real-valued functions, defined on abstract sets, with respect to (possibly infinite) non-negative measures, readers can see 
\cite{BCSHENSTOCK,BRV,CAO, AB,FM,GORDON,LEE,PFEFFER,DPHENSTOCK,DPM,RIECAN}
and the references among them.  Very recently, Boccuto et al. \cite{ABH} introduced the Kluv\'anek–Lewis ($(KL)-$)integral for Banach valued spaces. In \cite{Gould} Gould introduced the integral over vector-valued measures. Precupanu and Croitoru  \cite{AM,AM1} was generalized the gould type integral for Banach space valued functions over multimeasures. Gavrilu\c{t} \cite{Gavrilut} investigated a gould type integral with respect to a multisubmeasure. Recently, Candeloro et al. \cite{CDPMS0,CDPMS1} studied about integration theory of multifunctions. It has an important
role in applications to several branches of science, like for instance control theory, differential inclusions, game theory, aggregation functions, economics, problems of finding equilibria, and optimization. Since neither the Riemann
integral, nor the Lebesgue integral are completely satisfactory concerning the problem of the existence of primitives, different types of integrals extending the
previous ones have been introduced and investigated, like Henstock-Kurzweil, McShane, and Pettis integrals. These topics have many connections with measures taking values in abstract spaces, and in particular, the extension of the concept of integrability to set-valued functions can be used in order to obtain a larger number of selections for multifunctions, through their estimates and properties, in several applications (see \cite{CSSISY, CSPANAM}). In \cite{NS,Ns1} Papageorgious introduced the concept of multivalued integral over set-valued measure and the theory of set valued
functions and set valued measures. 
  Kandilakis \cite{Kandi} introduced an integral of bounded real-valued functions with respect to a set-valued measure using the set
of Kluv\'{a}nek-Knowles type integrals. Wu et al. \cite{W1} introduced the set-valued Bartle integral which is the set of Bartle-Dunford-Schwartz type integrals. The set-valued Bartle integral has also been deeply considered by Zhang et al. \cite{W} to introduced a multivalued gould type integral for a real-valued function with respect to a set-valued measure (see \cite{AM,AM1}). In \cite{Cai}, Cai Li Zhou and  Fu Gui Shi introduced and studied  a new set-valued integral of scalar-valued functions with respect to a set-valued measure. Unlike various ways of defining such an integral (e.g., Papageorgiou way, Kandilakis way,  Wu,
Zhang, and Wang way), their integral does not depend on measure selections of the given set-valued measure and is a Pettis type set-valued integral of scalar-valued functions with
respect to a set-valued measure. The result is an extension of
the Kluv\'{a}nek-Lewis type integral, which is a Pettis type weak integral of scalar-valued functions with respect to a vector-valued measure.\\

 In this paper we study some properties and convergence theorem  of  Kluv\'{a}nek-Lewis-Henstock $\m-$integrability for  $\m-$measurable functions. We introduce   Kluv\'{a}nek-Lewis-Henstock integrable of scalar-valued functions with respect to a set valued measure in a Banach space. Finally we introduce Lebesgue type dominated convergence theorem for the set-valued Kluv\'{a}nek-Lewis-Henstock integral.
\section{Preliminaries and Definitions}
Let $T \neq \emptyset$ be an abstract set, $P[T]$  be the class of all subsets of $T,~\Sigma \subset P[T]$ be a $\sigma-$algebra, $X$ be a Banach space, and $X^{\ast}$ be its topological dual. For each $ A \in \Sigma,$ let us denote by $\chi_A$  the \emph{characteristic function} of $A$ defined by \begin{eqnarray*}
\chi_A(t)=\left\{
\begin{array}{ll}
1 & \text{if  } t \in A, \\ \\
0 & \text{if  } t \in T \setminus A.
\end{array}\right. 
\end{eqnarray*}
Assume $T $ is a compact topological space and $\Sigma $ is the $\sigma-$algebra of all Borel subsets of $T.$ Let $\nu:\Sigma \to \mathbb{R}
\cup \{+ \infty\}
$ be a $\sigma$-additive non-negative measure. A \emph{vector measure} is a
$\sigma$-additive set function $\mu: \Sigma \to X.$ 
By the Orlicz-Pettis theorem (see also 
\cite[Corollary 1.4]{JD}), 
the $\sigma$-additivity of $\mu$ is equivalent to the
$\sigma$-additivity of 
the scalar-valued set function 
$x^{\ast} \mu : A \mapsto x^{\ast}(\mu(A))$ on 
$\Sigma $ for every $x^{\ast} \in X^{\ast}$.
For a literature on vector measures, see also 
\cite{CG,CURBERA,OP,JD,KA,KK,DR,OKADA,PANCHA,S} 
and the references therein. To make our presentation reasonably self-contained, we recalling few definitions and results in this section that we will use in our main section.

The \emph{variation} $|\mu|$ of $\mu$ is defined by 
\begin{eqnarray*}\label{variation}
|\mu|(A)= \sup \left\{
\sum_{i=1}^r \|\mu(A_i)\|: A_i \in \Sigma,
i=1,2, \ldots, r; 
A_i \cap A_j = \emptyset \text{  for  }
i \neq j ; \, \bigcup_{i=1}^r A_i \subset A
\right\}.
\end{eqnarray*}
We define the \emph{semivariation} $\|\mu\|$ of $\mu$ by
\begin{eqnarray}\label{domination}
\|\mu\|(A)= \sup_{x^{\ast} \in X^{\ast},
\|x^{\ast}\| \leq 1} |x^{\ast} \mu|(A).
\end{eqnarray}
\begin{rem}\label{finitesemivariation}
\rm Observe that $\|\mu\|(A) < + \infty$ for all 
$A \in \Sigma$ (see also \cite[Corollary 1.19]{JD}, 
\cite[\S1]{OKADA}).
\end{rem}
The \emph{completion} of $\Sigma$ with respect to
$\|\mu\|$ is defined by 
\begin{eqnarray}\label{completion}
\widetilde{\Sigma}= \{A= B \cup N: B \in \Sigma, 
N \subset M \in \Sigma \text{  with  } \|\mu\|(M)=0\}.
\end{eqnarray}
A function $f:T \to \mathbb{R}$ is said to be 
\emph{$\mu$-measurable} if 
\begin{eqnarray*}\label{mumeasurability}
f^{-1}(B) \cap \{ t \in T: f(t) \neq 0 \}\in 
\widetilde{\Sigma} 
\end{eqnarray*}
for each Borel subset $B \subset \mathbb{R}.$

Observe that from (\ref{domination}) and (\ref{completion})
it follows that every $\mu$-measurable real-valued
function is also $x^{\ast} \mu$-measurable 
for every $x^{\ast} \in X^{\ast}.$ 
Moreover, it is readily seen that every 
$\Sigma$-measurable real-valued function is also
$\mu$-measurable.  The complement of $\sigma(P)$ is $\overline{\sigma(P)} $ with respect to the semi variation $||\mu||.$  We recalling the definition of $\overline{\sigma(P)} $ as follows:
\begin{Def}
\cite[Definition 11]{DP2} 
$\overline{\sigma(P)}=\{E = F \cup N:~F \in \sigma(P),~N \subset M \in \sigma(P)~{\textit{with}}~||\mu||=0\}.$
\end{Def}
Recalling that $||\mu||$ is a $\sigma-$subadditive submeasure on $\sigma(P),$ then $\overline{\sigma(P)}$ is a $\sigma-$ring and that $\overline{\sigma(P)}$ contains $\sigma(P).$ 
We say that $\mu$ is \emph{$\Sigma$-separable} (or
\emph{separable}) if there is 
a countable family $\mathbb{B}=(B_k)_k$ in $\Sigma$ such that,
for each $A \in \Sigma$ and $\varepsilon > 0$, there is
$k_0 \in \mathbb{N}$ such that
\begin{eqnarray}\label{sep}
\|\mu\|(A \Delta B_{k_0})=
\sup_{x^{\ast} \in X^{\ast}, 
\|x^{\ast}\| \leq 1}
[\,|x^{\ast} \mu |(A \Delta B_{k_0}) ] \leq \varepsilon
\end{eqnarray}
(see also \cite{RICKER}). Such a family $\mathbb{B}$ is said 
to be \emph{$\mu$-dense}.

Observe that $\mu$ is $\Sigma$-separable if and only if 
$\Sigma$ is \emph{$\mu$-essentially countably generated},
namely there is a countably generated 
$\sigma$-algebra $\Sigma_0 \subset \Sigma$ such that
for each $A \in \Sigma$ there is $B \in \Sigma_0$ with
$\mu(A \Delta B)=0$. The separability of $\mu$ 
is satisfied, for instance, when 
$T$ is a separable metrizable space, $\Sigma$ is the
Borel $\sigma$-algebra of the Borel subsets of $T$, 
and $\mu$ is a Radon measure
(see also \cite[Theorem 4.13]{BREZIS}, 
\cite[Theorem 1.0]{DZAMONJA},
\cite[\S1.3 and \S2.6]{KK},
\cite[Propositions 1A and 3]{RICKER}). 
A \emph{decomposition} of a set 
$A \in \Sigma$ is a finite collection
$\{(A_1, \xi_1)$, $(A_2,\xi_2),
\ldots, (A_N, \xi_N)\}$ such that $
A_j \in \Sigma$ and $\xi_j \in A_j$ 
for every $j \in \{1$, $2, \ldots, N\},$
and $\nu(A_i \cap A_j)=0,$ whenever $i \neq
j$. A decomposition of subsets of $A \in \Sigma$ is called a \emph{partition} of $A$ when $\displaystyle{\bigcup_{j=1}^N A_j=A}.$ A \emph{gauge} on a set $A \in \Sigma$ is a map $\delta$
assigning to each point $x \in A$ a neighborhood $\delta(x)$ of
$x$. If ${\mathcal D}=\{(A_1, \xi_1)$, $(A_2,\xi_2),
\ldots, (A_N, \xi_N)\}$ is a
decomposition of $A$ and $\delta$ is a gauge on $A$, then we say that
${\mathcal D}$ is \emph{$\delta$-fine} if $A_j \subset \delta(\xi_j)$
for any $j \in \{ 1$, $2, \ldots, N \}.$

 Any gauge in $[a, + \infty]$ has the form
$\delta(x)=(x - d(x), x + d(x)),$ if $x \in [a, + \infty] \cap
\mathbb{R},$ and 
$\delta(+\infty)=(b, + \infty] = (b, + \infty) \cup
\{ + \infty \},$ where $d$ denotes a positive real-valued function
defined on $[a, + \infty).$ 
Now we define the \emph{Riemann sums} by
$\displaystyle{
S(f,{\mathcal D})= \sum_{j=1}^N \, f(\xi_j) \nu (A_j)}$ 
if the sum exists in $\mathbb{R},$ with the convention 
$0 \cdot (+\infty)=0.$
Note that for any
gauge $\delta$ 
there exists at least one $\delta$-fine partition
${\mathcal D}$ such that $S(f,{\mathcal D})$ is well-defined.
A \emph{vector measure} is a
$\sigma$-additive set function $\mu: \Sigma \to X.$ 
By the Orlicz-Pettis theorem (see also 
\cite[Corollary 1.4]{JD}), 
the $\sigma$-additivity of $\mu$ is equivalent to the
$\sigma$-additivity of 
the scalar-valued set function 
$x^{\ast} \mu : A \mapsto x^{\ast}(\mu(A))$ on 
$\Sigma $ for every $x^{\ast} \in X^{\ast}.$
\begin{Def}
\cite{ABH,HB}
A function $f:T \to \mathbb{R}$ is said to be 
\emph{Henstock-Kurzweil
integrable} (briefly, \emph{$(HK)$-integrable})
on a set $A\in \Sigma$ if there is an element
$I_A \in \mathbb{R}$ such
that for every $\varepsilon >0$ there is a gauge $\delta$ on
$A$ with
$| S(f, {\mathcal D}) - I_A| \leq \varepsilon,$
whenever ${\mathcal D}$ is a $\delta$-fine partition of $A$
such that $S(f, {\mathcal D})$ exists in $\mathbb{R}$, 
and we write $$\displaystyle{(HK)
\int_A \, f \, d\nu}=I_A.$$ 
 \end{Def}
\begin{Def}
\cite{ABH}
We say that
a $\Sigma$-measurable function $f:T \to \mathbb{R}$ 
is  
\emph{Kluv\'{a}nek-Lewis-Henstock-Kurzweil 
$\mu$-integrable},
shortly (\emph{$(HKL)$ $\mu$-integrable}) if 
the following properties hold: 
\begin{eqnarray}\label{first}
f \text{   is   } 
 |x^{\ast} \mu| 
\text{-Henstock-Kurzweil integrable  for each }  
x^{\ast} \in X^{\ast},
\end{eqnarray} and for every $A \in \Sigma$ there is
 $x_A^{(HK)} \in X$ with
\begin{eqnarray}\label{second} 
 x^{\ast}(x_A^{(HK)})=
(HK) \int_{A}f \, d |x^{\ast}\mu|
\text{    for  all  }
x^{\ast} \in X^{\ast}, 
\end{eqnarray}
where the symbol $(HK)$ in (\ref{second})
denote the usual  Henstock-Kurzweil integral
of a real-valued function with respect to an 
(extended) real-valued measure.
\end{Def}
A $\Sigma$-measurable function $f:T \to \mathbb{R}$ 
is said to be 
\emph{weakly $(HKL)$} \emph{$\mu$-integrable} if it 
satisfies only condition (\ref{first}).
If $f:T \to \mathbb{R}$ is a $(HKL)$-integrable
function, then the \emph{Alexiewicz norm} of $f$ is defined by 
\begin{eqnarray*}\label{alexiewicz}
||f||_{HKL}= \sup_{x^{\ast} \in X^{\ast}, 
\|x^{\ast}\|\leq 1} \left( \sup_{A \in \Sigma} 
\left| (HK) \int_A f(t) \, d|x^{\ast} \mu| \, \right| \, \right)
\end{eqnarray*} 
(see also \cite{ SWARTZ}).
Observe that, by arguing analogously as in
\cite[Theorem 9.5]{GORDON} and  
\cite[Example 3.1.1]{MCLEOD},  
for each $x^{\ast} \in X^{\ast}$
we get that $f=0$ 
$|x^{\ast} \mu|$-almost everywhere if and only if 
$\displaystyle{(HK) \int_A f(t) \, d|x^{\ast} \mu|}=0$
for every $A \in \Sigma.$ 
Thus, it is not difficult to see that $\|\cdot \|_{HKL}$ is a norm.
In general, the space of the real-valued
Henstock-Kurzweil integrable functions
endowed with the Alexiewicz norm is not complete.
\begin{thm}\label{may2}
\cite[Corollary 1.4]{JD}
\emph{(Orlicz-Pettis)} Let $X$ be a Banach space. Then a series $\sum\limits_{n=1}^{\infty}x_n $ in $X$ is unconditionally norm convergent whenever it is weakly subseries convergent.
\end{thm}
\begin{prop}\label{may02}
\cite[Proposition 1.1.1]{PANCHA}
Let $\mcR$ be a ring of sets in $T$ and let $\m: \mcR \to X $ be additive. Then $$||\m||(E)= \sup_{|x^{\ast}| \leq1 }\nu(x^{\ast} \m)(E)~{\textit{for}}~ E \in \sigma(\mcR).$$
\end{prop}
\begin{prop}\label{may002}
\cite[Proposition 5]{Dinu}
Let $P$ be a $\delta-$ring of subsets of $T$ and let $\nu$ be a ($\sigma-$additive) scalar measure on $P.$ If $f: T \to \mathbb{K}$ is $\nu-$integrable, let $\nu(E)= \int_{E}f d\nu,~E \in \sigma(P)$ then $\nu$ is $\sigma-$additive on $\sigma(P),$ where $\mathbb{K}$ is $\mathbb{C}$ or $\mathbb{R}.$
\end{prop}
\begin{prop}\label{0002}
\cite[Theorem 1.3]{DR}
Let $\Sigma$ be a $\sigma-$ring of sets. If $\nu: \Sigma \to X $ is $\sigma-$additive and if $(E_n)$ is a convergent sequence in $\Sigma$ with its limit $E,$ then $\lim\limits_{n}||\nu||(E_n)= ||\nu||(E).$
\end{prop}
\begin{thm}
\emph{(Pettis)} Let $\Sigma $ be a $\sigma-$ring of subsets of $T.$ Let $\lambda: \Sigma \to [0, \infty] $ be a $\sigma-$subadditive submeasure and let $\nu: \Sigma \to X $ be a $\sigma-$additive then $\nu << \lambda$ if and only if $\nu $ is $\lambda-$continuous.
\end{thm}
\begin{prop} 
\cite[Proposition 1.1.12]{PANCHA}
 A continuous submeasure $\la$ defined on a $\sigma-$ring $\mathcal{R}$ is a $\sigma-$subadditive.
\end{prop}
Throughout, we denote by $\sigma(.,A)$ the support function of the set $A \subset X$ defined by 
\begin{eqnarray*}
\sigma(x^\ast, A)= \sup\limits_{ x \in A}x^\ast(x),~x^\ast \in X^\ast
\end{eqnarray*}
 Let $P_0(X)=\{A \subset X:~A~is~non~empty~subset~of~X\}.$ Let $A,B \in P_0(X),$ the Hausdorff metric $H$ of $A~\&~B$ is defined by 
\begin{eqnarray*}
H(A,B)= \max\{\sup\limits_{a \in A}d(a,B),~\sup\limits_{b \in B}d(b,A)\},
\end{eqnarray*}
where $d(a,B)= \inf\limits_{b \in B}\{\sup d(a,B),~\sup\limits_{b \in B}d(b,A)\}$  and  $d(a,B)= \inf\limits_{b \in B}||a-b||.$
\begin{Def}
\cite[Definition 1]{Cai} Let $(T, \Sigma)$ be a measurable space. A set-valued set function $\mathcal{M}:\Sigma \to P_{0}(X)$ is said to be a set valued measure if it satisfies the following two conditions:
\begin{enumerate}
\item $\mathcal{M}(\emptyset)= \{0\}$
\item If $A_1, A_2,...$ are in $\Sigma$ with $A_i \cap A_j= \emptyset$ for $i \neq j,$ then $\mathcal{M}(\sum\limits_{n=1}^{\infty}A_n) = \sum\limits_{n=1}^{\infty}\mathcal{M}(A_n),$ where $$\sum\limits_{n=1}^{\infty}\mathcal{M}(A_n)= \{x \in X:~x= \sum\limits_{n=1}^{\infty}x_n~(unc.conv),~x_n \in \mathcal{M}(A_n),~n \geq 1 \}.$$ 
\end{enumerate}
\end{Def}
In particular $\mathcal{M}: \Sigma \to P_{wkc}(X)$ is a set-valued measure if and only if $\sigma(x^\ast, \mathcal{M}(.))$ is a scalar valued measure for all $x^\ast \in X^\ast.$ A vector valued measure $\mu:\Sigma \to X$ satisfying $\mu(A) \in \mathcal{M}(A)~\forall~A \in \Sigma$ is said to be a selection of the set valued measure $\mathcal{M}.$ We denote the set of all selections of $\mathcal{M}$ by $\mathcal{S}_{\mathcal{M}}.$

Recalling that the support function satisfies the property
\begin{equation*}
\sigma(x^\ast, A+B)= \sigma(x^\ast,A) + \sigma(x^\ast, B)
\end{equation*}
\begin{equation*}
\sigma(x^\ast, \lambda A)= \lambda( x^\ast, A)
\end{equation*}
for all $A, B \in P_{0}(X),~x^\ast \in X^\ast ~\&~\lambda\geq 0.$ In particular 
\begin{eqnarray*}
H(A,B)= \sup\limits_{||x^\ast||\leq 1}|\sigma(x^\ast,A)-\sigma(x^\ast,B)|,
\end{eqnarray*}
whenever $A,B$ are two convex sets.
\section{Basic properties of Kluv\'{a}nek-Lewis-Henstock integrable $\m-$ integral:}

In this section we will discuss basic properties of  (HKL)$\m-$integrable function which are $\m-$ measurable.
\begin{Def}\label{may001}
If $f $ is \emph{$\m-$measurable,} then there exists $M \in \sigma(P)$ with $||\m||(M)=0$ such that $f_{\chi_{T \setminus M}} $ is \emph{(HKL) $\m-$integrable } in $T,$ then we say that $f$ is \emph{(HKL) $\m-$integrable} in $T,$ and in that case we define \begin{eqnarray}
(HKL) \int_E f d|x^{\ast}\m|= (HKL)\int_E f_{\chi_{T \setminus M}}d |x^{\ast}\m| ~for ~E \in \sigma(P).
\end{eqnarray}
 and 
 \begin{eqnarray}
 (HKL)\int_T f d |x^{\ast} \m|= (HKL) \int_{\N}f \chi_{T \setminus M} d |x^{\ast} \m|.
 \end{eqnarray}
\end{Def}
Let $f$ be $\m-$measurable and let it be \emph{(HKL) $\m-$integrable} in $T$ with $M$ as in Definition \ref{may001}. If $M_1 \in \sigma(P)$ with $||\m||(M_1)=0 $ is such that $f_{\chi_{T \setminus M_1}}$ is\emph{ $\sigma(P)-$measurable}, then it can easily be shown that $f_{\chi_{T \setminus M_1}}$ is also \emph{(HKL)$\m-$integrable} in $T$ and $(HKL)\int_{E}f_{\chi_{T \setminus M}}d|x^{\ast}\m|$ for $E \in \sigma(P) \cup \{T\}.$ Consequently $(HKL)\int_{E}f d|x^{\ast}\m|$ is well defined and does not depend on the particular choice of the $|x^{\ast}\m|-$null set $M$ in Definition \ref{may001}. 
From now onward we denote the notation $|x^{\ast}\m| $ as $m.$ $\mcH(\m)$ denotes the class of all $\m-$measurable scalar functions on $T$
which are (HKL) $\m-$integrable in $T .$
It is observe that 
if $A$, $B\in \Sigma$,
$B \subset A$ and $f:T \to \mathbb{R}$ is 
$(HKL)$-integrable on $A,$  then $f$ is also
$(HKL)$-integrable on $B$ and on $A \setminus B,$ and 
\begin{eqnarray*}\label{add}
(HKL) \int_A f (t) \, dm= (HKL) \int_B f (t) \, dm +
(HKL) \int_{A \setminus B} f (t) \, dm .
\end{eqnarray*}
\begin{thm}\label{bhanish}
Let $f: T\to K$  or $[-\infty,\infty]$ and let $s$ be a $P$-simple function. 
Then the following hold:
\begin{enumerate}
\item  If $s = \sum\limits_{1}^{r} \alpha_ i \chi_{E _i}$ is \emph{ $P-$simple } with $(\alpha_i)_{1}^{r} \subset K , ~(E _i )_{1}^{r} \subset P$ and $E_i \cap E_j = \emptyset$ for $i \neq j, $ then $s$ is $\m-$integrable in $T$ and $(HKL)\int_{E}s d\m= \sum\limits_{1}^{r}\alpha_i \m (E_i \cap E) $ for $E \in \sigma(P).$ Consequently $$||\m||(E)= \sup\{|(HKL)\int_{E}s d m|~: {\textit{s}}~ P-{\textit{simple,}} ~|s(t)| \leq \chi_E(t),~t \in T\}.$$ 
\item If $f$ is (HKL)$\m-$integrable in $T,$ then $\nu(.): \sigma(P) \to X$ given by $$ \nu(.)= (HKL)\int_{(.)} f dm ~{\textit{is}} ~\sigma-{\textit{additive in norm}}.$$
\item If $\nu(.)$ as in (2), then 
\begin{enumerate}
\item[(a)] $||\nu||(E)= \sup_{|x^{\ast}| \leq 1}\int_{E}|f|d \nu(m),~E \in \sigma(P).$
\item[(b)] $\lim\limits_{||\m|| \to 0} \nu(E)= \lim\limits_{||\m||(E) \to 0}||\nu||(E)=0,~E \in \sigma(P).$
\end{enumerate}
\item $\mcH[\m]$ is a vector space over $K$ with respect to pointwise addition and scalar multiplication. Moreover, for $E \in \sigma(P)$ fixed, the mapping $f \to (HKL)\int_{E}f dm $ is linear on $\mcH(\m).$ Consequently $\int_{E} s d m = \sum\limits_{i=1}^{r} \alpha_i m(E_i \cap E) $ for $E \in \sigma(P) $ and for $s = \sum\limits_{1}^{r}\alpha_i \chi_{E_i},$ with $(E_i)_{1}^{r} \subset P $ not necessarily mutually disjoint.
\end{enumerate}
\end{thm}
\begin{proof}
Without loss of generality we shall assume that all the $\m-$measurable functions considered here are further $\sigma(P)-$measurable.\\
(1)  is very obvious. Using Theorem \ref{may2} we can prove (2).\\
3(a) Using $(2)$ and Proposition \ref{may02} and  Proposition \ref{may002}.\\
3(b) If $||\mu||(E)=0,~E \in \sigma(P)$ then $\nu(x^\ast \mu)(E)=0$ for each $x^\ast \in X^\ast.$ Now by part  3(a) we get $||\nu||(E)=0.$ As $||\mu||$ and $||\nu||$ are non negative, monotone and $\sigma-$subadditive on the $\sigma-$ring $\sigma(P)$ then $f$ is $(HKL)\mu-$integrable in $T$ implies $\nu(.)=(HKL)\int_{(.)}f d m$ is $\sigma-$additive in norm. Also by Proposition \ref{0002}, $||\nu||$ is continuous on $\sigma(P).$ From the fact $|\nu(E)| \leq ||\nu||(E)~,E \in \sigma(P)$ finally gives $$\lim\limits_{||\mu||\to 0}\nu(E)= \lim\limits_{||\mu||(E) \to 0}||\nu||(E)=0~,E \in \sigma(P).$$ 
(4) It is very obvious, so we leave for readers.
\end{proof}
\begin{thm}
 For a set $A \in T,$ let $P \cap A= \{B \cap A:~B \in P\}$. If $ f $ is an $\m-$measurable function on $T$ and if $f$ is $\m-$essentially bounded in a set $ A \in P,$ then $f $ is (HKL)-$\m-$integrable in each $B \in P \cap A$ and $$|(HKL)\int_{B}f d\m| \leq (ess \sup_{t \in A}|f(t)|)||\m||(B),~B \in \sigma(P) \cap A = P \cap A.$$ 
\end{thm}
\begin{proof}
Let $A \in P$ and $\mathfrak{K}= ess\sup\limits_{t \in A}|f(t)|$ then there exists $M_1 \in \sigma(P) \cap A= P \cap A$ with $||\m||(M_1)=0$ such that $\sup\limits_{t \in A \setminus M_1}|f(t)|= \mathfrak{K}.$ Let $\Sigma_{A}= P \cap A$ then $\Sigma_{A}$ is $\sigma$ algebra of subset of $A$ also $f_{\chi_{A\setminus M_1}}$ is $\Sigma_{A}-$measurable and bounded then there exists a sequence $(s_n)_{n=1}^{\infty}$ of $\Sigma_{A}-$simple function such that $s_n \to f_{\chi_{A\setminus M_1}}$ uniformly in $A$ with $|s_n| \to f$ in $A.$ Let $B \in \Sigma_{A}$ then for $\varepsilon>0$ there exists $n_0$ such that $||s_n -s_m||_{A}||\m||(A) < \varepsilon$ for $n, m > n_0.$ Consequently,
\begin{align*}
|(HK)\int_{B}s_n d\m -(HK)\int_{B}s_m d\m|&\leq ||s_n-s_m||_{B}||\m||(B)\\&\leq ||s_n -s_m||_{A}||\m||(A)\\&<\varepsilon~ \forall ~n,m \geq n_0.
\end{align*}
So, for each $B \in \Sigma_{A}$ there exists a vector $x_{B}^{(HK)} \in X$ such that $\lim\limits_{n \to \infty}(HK)\int_{B}s_n d\m = x_{B}^{(HK)}.$
By Lebesgue Dominated Convergence Theorem on $\sigma-$algebra $\Sigma_A$ for scalar measure $$x^\ast(x_{B}^{(HK)}= \lim\limits_{n \to \infty}(HK)\int_{B}s_n d(x^\ast \m)$$ for each $x^\ast \in X^\ast$ and for each $B \in \Sigma_A,$ thus $f$ is $(HKL)-\m$-integrable in $A$ and $$(HKL)\int_{B}f d\m = x_{B}^{(HK)},\mbox{~for~each~}B \in \Sigma_A.$$  So, 
\begin{align*}
|(HKL)\int_{B}f d\m| &=|x_{B}^{(HK)}|\\&=|\lim\limits_{n \to \infty}(HK)\int_{B}s_nd \m|\\&\leq \mathfrak{K}||\m||(B)
\end{align*}
for $B \in \Sigma_A.$ Hence the complete proof.
\end{proof}
\begin{thm}
If $\phi $ is an $\m-$-essentially bounded $\m-$measurable scalar function on $T$ and if $f \in \mcH(\m),$ then $\phi.f $ is (HKL) $\m-$integrable in $T$ and $$(HKL)\int_{E}f d(|x^\ast\m|) = (HKL)\int_{T}\chi_E f d(|x^\ast \m|).$$
\end{thm}
\begin{proof}
Let $ess\sup\limits_{t \in T}|\phi(t)|= \mathfrak{K}$ and let $\mu(.)= (HKL)\int_{(.)}f|x^\ast \mu|$ on $\sigma(P),$ then $\mu$ is $\sigma-$additive on $\sigma(P)$ and $\sigma(P)-$measurable function $\phi$ is also $\mu-$essential bounded hence $\phi$ is $(HKL)\mu-$integrable in $T.$ Let $(HKL)\int_{E}\phi d \mu = x_{E}^{(HK)} $ for $E \in \sigma(P).$ Now
\begin{align*}
x^\ast(x_{E}^{(HK)}&=(HK)\int_{E}\phi d (|x^\ast \mu|)\\&= (HK)\int_{E}\phi f d (|x^\ast \mu|)
\end{align*}
for $E \in \sigma(P)$ and $x^\ast \in X^\ast.$ Using \cite[Proposition 1.1.18]{PANCHA}, Lebesgue Dominated Convergence Theorem for $(KL)-$integral and \cite[Theorem 3]{Dinu} it can be extended to a general $\sigma(P)-$measurable $x^\ast \mu-$integrable function $\phi.$ Hence $\phi.f$ is $(HKL)\mu-$integrable in $T$ also if $E \in \sigma(P)$ and $f$ is $(HKL)\mu-$integrable in $T$ that is if $(HKL)\int_{E}f d |x^\ast \mu|= x_{E}^{(HK)} \in X,$ then
\begin{align*}
x^\ast(x_{E}^{(HK)})&= (HK)\int_{E}f d (|x^\ast \mu|)\\&= (HK)\int_{T} \chi_E f d(|x^\ast \mu|)
\end{align*}
for $x^\ast \in X^\ast.$
\end{proof}
\begin{thm}
If $ u \in \mcH(X,Y), $ then $u \m: P\to X$ is a $\sigma-$additive and $f$ is (HKL)$u \m-$integrable in $T$ whenever $f$ is (HKL) $\m-$integrable in $T.$ In that case, $$ u(HKL) \int_{E} f d\m = (HKL) \int_{E} f d( u \m) {\textit{ for}}~ E \in \sigma(P).$$
\end{thm}
\begin{proof}
Since $u\mu: P \to X$ is $\sigma-$additive so, $u\mu: P \to Y$ is $\sigma-$additive for $y^\ast \in Y^\ast,~y^\ast u \in X^\ast$ and hence the proof.
\end{proof}
\subsection{  $\m$-a.e. version of convergences}
\begin{thm}\label{bhanish1}
\emph{(Dominated Convergence Theorem $\m$-a.e. convergence version)} Let $f_n,~n \in \N$ be $\m-$measurable with values in $K$ or $[-\infty, \infty] $ and let $g: T \to K$ be (HKL) $\m-$integrable in $T.$ Suppose $f_n(t)| \leq |g(t)|~\m-$a.e. in $T$ for all $n.$ If $f: T \to K$ and if $f_n(t) \to f(t) ~\m-$a.e. in $T$ then $f, ~f_n, ~n \in \N $ are (HKL)~$\m-$integrable in $T$ and $(HKL)\int_{E} fd \m = \lim\limits_{n \to \infty}(HKL)\int_{E} f_n d\m,~E \in \sigma(P),$ where the limit is uniform with respect to $ E \in \sigma(P).$ Moreover, $$ \lim\limits_{n \to \infty}\sup_{|x^{\ast}| \leq 1} \int_{T}|f_n - f|d|\m|(|x^\ast \mu|)=0.$$
 \end{thm}
\begin{proof}
Let us consider that all the $\mu-$measurable functions are $\sigma(P)$ measurable. Let $\nu(E)=(HKL)\int_{E}gd|x^\ast \mu|,~E \in \sigma(P), $ then by  Theorem \ref{bhanish}(2), $\nu$ is $\sigma-$additive on $\sigma(P).$ Let $f_n,~n \in \mathbb{N}$ are $(HKL)-\mu-$integrable in $T.$ As $f_n \to f~\mu-$a.e. in $T$ then clearly $f$ is $\overline{\sigma(P)}$ measurable. This means $f$ is $\sigma(P)-$measurable. Moreover, $|f|\leq |g|~\mu-$a.e. in $T.$ From the \cite[Theorem 2.1.5(vii)]{PANCHA}, $f$ is $(HKL)\mu-$integrable in $T$  then there exists $M \in \sigma(P)$ with $||\mu||(E)=0$ such that $g$ is finite on $T\setminus M$ and $f_n(t) \to f(t)$ for $t \in T \setminus M.$ Let $F= \bigcup\limits_{n=1}^{\infty}N(f_n) \cap (T \setminus M)= \bigcup_{n=1}^{\infty}(f_n \chi_{T \setminus M})$ then $F \in \sigma(P)$ and $f_n \chi_{T \setminus M} \to f \chi_{T \setminus M}$ pointwise in $T.$ Using the \cite[Corollary 5]{JD}, $\nu$ has a control measure $\mu_0: \sigma(P) \to [0, \infty).$ Thus for $\varepsilon>0$ there exists a $\delta>0$ such that $\mu_0(E) < \delta$ implies $||\nu||(E)< \varepsilon.$ Now by Egoroff-Lusin Theorem there exists a set $N \in \sigma(P) \cap F$ with $\mu_0(N)=0$ and a sequence $\{F_k\}_{k=1}^{\infty} \subset P$ with $F_k $ converges in $F \setminus N$ such that $f_n \to f$ uniformly in each $F_k.$ As $\mu_0(F \setminus N \setminus F_k) \to 0$ there exists $k_0$ such that $\mu_0(F \setminus N \setminus F_k) < \delta$ for $k \geq k_0.$ Since $\mu_0(N)=0,~||\nu||(N)=0$ and $||\mu||(M)=0$ therefore $||\nu||(M)=0$ and  since$f_n \to f$ uniformly in $F_{k_0}$ then there exists $ n_0$ such that $||f_n -f||_{F_{k_0}}||\mu||(F_{k_0}) < \varepsilon~\forall ~n \geq n_0.$ For $E \in \sigma(P)$ and for $|x^\ast|\leq 1$ with \cite[ Proposition 4 \S 8]{Dinu}  with the fact $\mu_0(F \setminus N \setminus F_{k_0})<\delta$ and $||\nu||(N)= ||\nu||(M)=0$ we get
\begin{align*}
&\quad |x^\ast(HKL)\int_{E}f_n d \m-x^\ast(HKL)\int_{E}f d \mu|\\
&=|(HK)\int_{E}(f_n -f)d (|x^\ast \mu|)\\&\leq (HK)\int_{E}|f_n-f|d|\mu|(|x^\ast \mu|)\\&\leq (HK)\int_{E \cap F_{k_0}}|f_n-f|d|\mu|(|x^\ast \mu|)+(HK)\int_{E \cap (T \setminus F_{k_0})}|f_n -f|d |\mu|(|x^\ast \mu|)
\end{align*}
Finally from \cite[Proposition 8 \S 8]{Dinu} and the Theorem \ref{bhanish}(3) we have 
\begin{align*}
|(HKL)\int_{E}f_n d \mu -(HKL)\int_{E}f d\mu| &\leq \sup\limits_{|x^\ast|\leq 1}(HK)\int_{E}|f_n-f|d|\mu|(|x^\ast \mu|)\\&\leq \varepsilon
\end{align*}
for $n \geq n_o$ and for all $E \in \sigma(P).$
\end{proof}
\begin{thm}
\emph{(Bounded Convergence Theorem $\m$-a.e. convergence version)}: Let $P$ be a $\sigma-$ring. Let $f_n, n \in \N,$ be $\m-$measurable with values in $K$ or in $[-\infty, \infty]$ and let $K$ be a positive constant such that $|f_n(t)| \leq K~\m-$a.e. in $T$ for all $n.$ If $f: T \to K $ and $f_n(t) \to f(t)~\m-$a.e. in $T$ then $f, f_n , n \in \N$ are (HKL) $\m-$integrable in $T $ and $(HKL) \int_{E} f d m = \lim\limits_{n}(HKL)\int_{E} f_n d m ~{\textit{for}}~ E \in S,$  the limit being uniform with respect to $E \in S.$
\end{thm}
\begin{proof}
Since $P$ is a $\sigma-$ring and as the  constant functions are $(HKL)\mu-$integrable in $T$ so the proof follows from  Theorem \ref{bhanish1}.
\end{proof}
\section{Set valued Kluv\'{a}nek-Lewis-Henstock integral}
Let $HK(T, \mathbb{R}, |\mathcal{M}|)$ be the space of all functions $f: T \to \mathbb{R}$ which are $\Sigma-$measurable and $|\mathcal{M}|-$integrable.
\begin{Def}
Let $\mathcal{M}: \Sigma \to P_{wkc}(X)$ be a set-valued measure and $f: T \to \mathbb{R}$ an element of $HK(T, \mathbb{R}, |M|),~f~$ is said to be \emph{Kluv\'{a}nek-Lewis-Henstock integrable} with respect to $\mathcal{M}~$ (for short) $(KHL)\mathcal{M}-$integrable if 
\begin{enumerate}
\item $f$ is $\sigma(x^\ast, \mathcal{M}(.))-$integrable for each $x^\ast \in X^\ast.$
\item For each $A \in \Sigma,$ there exists a $W_{A}^{HK} \in P_{wkc}(X)$ such that 
\begin{eqnarray*}
\sigma(x^\ast, W_{A}^{HK})= (HK)\int_{A} f(w)d \sigma(x^\ast, \mathcal{M}(w)) \mbox{~for~each~}x^\ast \in X^\ast. 
\end{eqnarray*}
\end{enumerate}
In this case, we write $W_{A}^{HK}= (HKL)\int_{A}f(w)d \mathcal{M}(w)$ for each $A \in \Sigma$ and call it set-valued \emph{Kluv\'{a}nek-Lewis-Henstock integral} of $f$ with respect to $\mathcal{M}$ on $A.$
\end{Def}
\begin{rem}
The set-valued \emph{Kluv\'{a}nek-Lewis-Henstock integral} is generalization of the \emph{Kluv\'{a}nek-Lewis-Henstock integral}. When the set-valued measure $\mathcal{M}$ is degenerated into a vector-valued measure $\mu$ the set $W_{A}^{HK}$ is reduced to a vector in $X$ and the equality 
\begin{eqnarray*}
\sigma(x^\ast, W_{A}^{HK})=(HK)\int_{A}f(w)d \sigma(x^\ast, \mathcal{M}(w))
\end{eqnarray*}
turns into 
\begin{eqnarray*}
x^\ast(W_{A}^{HK})=(HK)\int_{A}f(w)dx^\ast\mu(w) \mbox{~for~each~}x^\ast \in X^\ast.
\end{eqnarray*}
\end{rem}
\begin{ex}
Let $\mathcal{M}: \Sigma \to P_{wkc}(X)$ be a set-valued measure. If $f= \sum\limits_{i=1}^{n}X_i \xi_{A_i}$ is a non-negative simple function, then $f$ is $(HKL) \mathcal{M}-$integrable and 
\begin{eqnarray}
(HKL)\int_{A} f(w)d\mathcal{M}(w)= \sum\limits_{i=1}^{n}\lambda_i(\mathcal{M}(A_i \cap A)) \mbox{~for~each~}A \in \Sigma.
\end{eqnarray}
\end{ex}
Some properties of the set-valued Kluv\'{a}nek-Lewis-Henstock integrals are 
\begin{thm}
Let $\mathcal{M}:\Sigma \to P_{wkc}(X)$ be a set-valued measure, $f,g: T \to \mathbb{R}(HKL)M-$integrable and $\lambda \geq 0.$ Then 
\begin{enumerate}
\item $f+g$ is $(HKL)\mathcal{M}-$integrable and 
\begin{eqnarray*}
(HKL)\int_{A}(f(w)+g(w))d\mathcal{M}(w)= (HKL)\int_{A}f(w)d \mathcal{M}(w)+ (HKL)\int_{A}g(w)d \mathcal{M}(w) \mbox{~for~each~} A \in \Sigma.
\end{eqnarray*}
\item $\lambda f$ is $(HKL)\mathcal{M}-$integrable and $(HKL)\int_{A}\lambda f(w)d \mathcal{M}(w)= \lambda (HKL)\int_{A}f(w)d\mathcal{M}(w)$ for each $A \in \Sigma.$
\end{enumerate}
\end{thm}
\begin{proof}
(1) Since $f$ and $g$ are $(HKL)\mathcal{M}-$integrable, $f$ and $g$ are $\sigma(x^\ast, \mathcal{M}(.))-$integrable for each $x^\ast \in X^\ast$ and there exist $(HKL)\int_{A}f(w)d\mathcal{M}(w),~(HKL)\int_{A}g(w)d\mathcal{M}(w)$ in $P_{wkc}(X)$ such that 
\begin{eqnarray*}
\sigma(x^\ast, (HKL)\int_{A}f(w)d\mathcal{M}(w))&= \int_{A}f(w)d\sigma(x^\ast, \mathcal{M}(w))\\
\sigma(x^\ast, (HKL)\int_{A}g(w)d\mathcal{M}(w))&= \int_{A}g(w)d\sigma(x^\ast, \mathcal{M}(w))
\end{eqnarray*}
for each $A \in \Sigma$ and $x^\ast \in X^\ast.$ Obviously, $ f+g$ is $\sigma(x^\ast, \mathcal{M}(.))-$integrable and by properties of support functions, we have 
\begin{align*}
&\quad \sigma(x^\ast, (HKL)\int_{A}f(w)d\mathcal{M}(w)+(HKL)\int_{A}g(w)d\mathcal{M}(w))\\
&= \sigma(x^\ast, (HKL)\int_{A}f(w)d\mathcal{M}(w))+\sigma(x^\ast, (HKL)\int_{A}g(w)d\mathcal{M}(w))\\&= (HK)\int_{A}f(w)d\sigma(x^\ast, \mathcal{M}(w))+(HK)\int_{A}g(w)d\sigma(x^\ast, \mathcal{M}(w))\\&=(HK)\int_{A}\{f(w)+g(w)\}d\sigma(x^\ast, \mathcal{M}(w))
\end{align*}
for each $A \in Sigma$ and $x^\ast \in X^\ast.$  \\
Since $(HKL)\int_{A}f(w)d\mathcal{M}(w),~(HKL)\int_{A}g(w)d\mathcal{M}(w) \in P_{wkc}(X).$ It follows that $f+g$ is $(HKL)\mathcal{M}-$integrable and \begin{eqnarray*}
(HKL)\int_{A}(f(w)+g(w))d\mathcal{M}(w)= (HKL)\int_{A}f(w)d \mathcal{M}(w)+ (HKL)\int_{A}g(w)d \mathcal{M}(w)
\end{eqnarray*}
for each $A \in \Sigma.$\\ 
(2) \begin{align*}
\sigma(x^\ast, \lambda(HKL)\int_{A}f(w)d\mathcal{M}(w))&= \lambda \sigma(x^\ast, (HKL)\int_{A}f(w)d\mathcal{M})\\&=\lambda (HK)\int_{A}f(w)d\sigma(x^\ast, \mathcal{M}(w))\\&=(HK)\int_{A}\lambda f(w)d\sigma(x^\ast, \mathcal{M}(w)).
\end{align*}
For each $x^\ast \in X^\ast,~A \in \Sigma$ and $\lambda \geq 0.$ Since $(HKL)\int_{A}f(w)d\mathcal{M}(w) \in P_{wkc}(X),$ we have $\lambda (HKL)\int_{A}f(w)d\mathcal{M}(w) \in P_{wkc}(X).$ Thus $\lambda f$ is $(HKL)\mathcal{M}-$integrable and $\lambda f$ is $(HKL)\mathcal{M}-$integrable and $(HKL)\int_{A}\lambda f(w)d \mathcal{M}(w)= \lambda (HKL)\int_{A}f(w)d\mathcal{M}(w)$ for each $A \in \Sigma.$
\end{proof}
\begin{thm}
Let $(T, \Sigma, \mu)$ be a non negative finite measure space. If $\mathcal{M}:\Sigma \to P_{wkc}(X)$ is a $\mu-$continuous set valued measure and $f: T \to \mathbb{R}$ is $(HKL)\mathcal{M}-$integrable, then $\mathfrak{M}:\Sigma \to P_{wkc}(X)$ defined by $\mathfrak{M}(A)=(HKL)\int_{A}f(w)d\mathcal{M}(w)~\forall A \in \Sigma$ is $\mu-$continuous set-valued measure.
\end{thm}
\begin{proof}
Clearly 
\begin{align*}
\sigma(x^\ast, \mathfrak{M}(\emptyset))&=\sigma(x^\ast, (HKL)\int_{\emptyset}f(w)d\mathcal{M}(w))\\&=(HK)\int_{\emptyset}f(w)d \sigma(x^\ast, \mathcal{M}(w))\\&=0, {~for~each~}x^\ast \in X^\ast.
\end{align*}
Also, 
\begin{align*}
\sigma(x^\ast, \mathfrak{M}(A))&= \sigma(x^\ast, (HKL)\int_{A}f(w)d\mathcal{M}(w))\\&= (HK) \int_{A}f(w)d \sigma (x^\ast, \mathcal{M}(w)).
\end{align*}
If $\sigma(x^\ast, \mcM(.))$ is a $\mu-$continuous scalar-valued measure for each $x^\ast \in X^\ast,$ then for each $x^\ast \in X^\ast$ and $A \in \Sigma, $ we have
\begin{align*}
\sigma(x^\ast, \mathfrak{M}(A))&= \sigma(x^\ast, (HKL)\int_{A}f(w)d\mcM(w))\\&= (HK)\int_{A}f(w)d \sigma (x^\ast, \mcM(w)).
\end{align*}
Since $f$ is $\sigma(x^\ast, \mcM(w))-$integrable then $A \to (HK)\int_{A}f(w)d\sigma(x^\ast, \mcM(w))$ is countable additive. Let us assume $\{A_n\}_{n \in \mathbb{N}}$ be a sequence of pair wise disjoint elements of $\Sigma,$ then 
\begin{align*}
\sigma(x^\ast, \mathfrak{M}(\cup_{n=1}^{\infty}A_n) &= \sigma(x^\ast, (HKL)\int_{\cup_{n=1}^{\infty}A_n}f(w)d\mcM(w))\\&=(HK)\int_{\cup_{n=1}^{\infty}A_n}f(w)d\sigma (x^\ast, \mcM(w))\\&=\sum\limits_{n=1}^{\infty}(HK)\int_{A_n}f(w)d \sigma(x^\ast, \mcM(w))\\&= \sum\limits_{n=1}^{\infty} \sigma(x^\ast, (HKL)\int_{A_n}f(w)d \mcM(w))\\&= \sum\limits_{n=1}^{\infty}\sigma(x^\ast, \mathfrak{M}(A_n)).
\end{align*}
Therefore $\sigma(x^\ast, \mathfrak{M}(.))$ is a scalar-valued measure for each $x^\ast \in X^\ast.$ Hence $\mathfrak{M}:\Sigma \to P_{wkc}(X)$ is a continuous set-valued measure.
\end{proof}
\begin{thm}
Suppose $X$ has the RNP, $X^\ast$ is separable, $\mcM: \Sigma \to P_{wkc}(X)$ is a $\mu-$continuous set-valued measure of bounded variation and $f: T \to \mathbb{R}^{+}$ is $(HKL)\mcM-$integrable then there exists an integrably bounded set-valued function $X: T \to P_{wkc}(X) $ such that 
\begin{eqnarray*}
(HKL)\int_{T}f(w)d\mcM(w)=(HK)\int_{T}f(w)X(w)d \mu(w).
\end{eqnarray*}
\end{thm}
\begin{proof}
Since $X$ has the RNP, $X^\ast$ is separable, $\mcM: \Sigma \to P_{wkc}(X)$ is a $\mu-$continuous set-valued measure of bounded variation then there exists an integrably bounded set-valued function $X: T \to P_{wkc}(X)$ which is a Radon Nikod\'{y}m derivative of $\mcM$ with respect to $\mu,$ then $\mcM(A)=(HK)\int_{A}X(w)d \mu(w)$ for each $A \in \Sigma.$ It follows that 
\begin{eqnarray}\label{may}
\sigma(x^\ast, \mcM(A))= \sigma(x^\ast, (HK)\int_{A}X(w)d \mu(w))= (HK)\int_{A} \sigma(x^\ast, X(w))d \mu(w)
\end{eqnarray}
for each $x^\ast \in X^\ast$ and each $A \in \Sigma.$

Now from the fact that $\mcM: \Sigma \to P_{wkc}(X)$ is $\nu-$continuous set-valued measure if and only if $\sigma(x^\ast, \mcM(.))$ is $\mu-$continuous scalar valued measure for each $x^\ast \in X^\ast.$ This fact with (\ref{may}) gives $\sigma(x^\ast, X(w))$ is a Radon-Nikod\'{y}m derivative of $\sigma(x^\ast, \mcM(.))$ with respect to $\mu$ for each $x^\ast \in X^\ast.$ Therefore 
\begin{eqnarray}\label{may1}
\sigma(x^\ast, (HKL)\int_{T}f(w)d \mcM(w))&= (HK)\int_{T}f(w)d \sigma(x^\ast, \mcM(w))\\&= (HK) \int_{T}f(w)\sigma(x^\ast, X(w))d \mu(w)
\end{eqnarray}
for each $x^\ast \in X^\ast.$ Since $f$ is non-negative from \cite[Theorem 8]{Cai}, we have 
\begin{align*}
(HK)\int_{T}f(w)\sigma(x^\ast, X(w))d\mu (w)&= (HK)\int_{T}\sigma(x^\ast, f(w)X(w))d\mu(w)\\&= \sigma(x^\ast, (HK)\int_{T}f(w)X(w)d\mu(w)
\end{align*}
for each $x^\ast \in X^\ast.$ From (\ref{may1}) 
\begin{eqnarray*}
(HKL)\int_{T}f(w)d\mcM(w)=(HK)\int_{T}f(w)X(w)d \mu(w).
\end{eqnarray*}
\end{proof}
\begin{thm}\label{may01}
Let $\mcM: \Sigma \to P_{wkc}(X)$ be a set-valued measure and let $f: T \to \mathbb{R}$ be $(HKL)\mcM-$integrable then 
\begin{eqnarray*}
|(HKL)\int_{A}f(w)d\mcM(w)| \leq (HK)\int_{A}|f(w)|d(w)
\end{eqnarray*}
for each $A \in \Sigma.$
\end{thm}
\begin{proof}
As $f$ is $(HKL)\mcM-$integrable, there exists $(HKL)\int_{A}f(w)d \mcM(w) \in P_{wkc}(X)$ such that 
\begin{align*}
\sigma(x^\ast, (HKL)\int_{A}f(w)d\mcM(w))&=(HK)\int_{A}f(w)d \sigma(x^\ast, \mcM(w))
\end{align*}
for each $A \in \Sigma$ and $x^\ast \in X^\ast.$ From the fact of support function
\begin{align*}
|(HKL)\int_{A}f(w)d\mcM(w)|&=H((HKL)\int_{A}f(w)d\mcM(w), \{0\})\\&=\sup\limits_{||x^\ast||\leq 1}|\sigma(x^\ast, (HKL)\int_{A}f(w)d\mcM(w)|\\&= \sup\limits_{||x^\ast||\leq 1}|(HK)\int_{A}f(w)d\sigma(x^\ast, \mcM(w)|\\&\leq \sup\limits_{||x^\ast||\leq 1}(HK)\int_{A}|f(w)|d|\sigma(x^\ast, \mcM(.)|(w).
\end{align*}
Now,
\begin{align*}
|\sigma(x^\ast, \mcM(0))|(A)&= \sup\sum\limits_{i=1}^{n}|\sigma(x^\ast, \mcM(A_i)|\\&\leq \sup\sum\limits_{i=1}^{n}||x^\ast|||\mcM(A_i)|\\&\leq \sup\sum\limits_{i=1}^{n}|\mcM(A_i)|\\&=|\mcM|(A)
\end{align*}
for each $A \in \Sigma$ and $x^\ast \in X^\ast$ with $||x^\ast||\leq 1,$ where the supremum is taken over all finite measurable partition $\{A_1, A_2,..,A_n\}$ of $A,$ as $f$ is non negative so $$(HK)\int_{A}|f(w)|d|\sigma(x^\ast, \mcM(.))|(w)\leq (HK)\int_{A}|f(w)d|\mcM|(w)$$ for each $A \in \Sigma $ and $x^\ast \in X^\ast$ with $||x^\ast||\leq 1.$ Therefore 
\begin{align*}
|(HKL)\int_{A}f(w)d \mcM(w)| &\leq \sup\limits_{||x^\ast||\leq 1}(HK)\int_{A}\int_{A}|f(w)|d|\sigma(x^\ast, \mcM(w))|\\&\leq (HK)\int_{A}|f(w)|d|\mcM|(w)\\&\leq (HK)\int_{A}|f(w)|d(w)
\end{align*}
for each $A \in \Sigma.$
\end{proof}
\begin{cor}
Let $\mcM: \Sigma \to P_{wkc}(X)$ be a set-valued measurable and let $f: T \to \mathbb{R}$ be $(HKL)\mcM-$integrable then 
\begin{enumerate}
\item $\lim\limits_{A \to \emptyset}H((HKL)\int_{A}f(w)d\mcM(w), \{0\})=0.$
\item If $f= \emptyset$ then $(HKL)\int_{A}f(w)\mcM(w)=\{0\}.$
\item If $f=0$ then $(HKL)\int_{A}f(w)d\mcM(w)=\{0\}$ for each $A \in \Sigma.$
\end{enumerate}
 \end{cor}
\begin{thm}
Let $\mcM: \Sigma \to P_{wkc}(X)$ be a set-valued measure and let $f, ~f_n: T \to \mathbb{R},~n \in \mathbb{N}$ be $(HKL)\mcM-$integrable such that $\{f_n\}_{n \in \mathbb{N}}$ is uniformly integrable with respect to $|\mcM|$ and $\lim\limits_{n \to \infty}f_n =f,~|\mcM|-$a.e., then 
\begin{eqnarray*}
\lim\limits_{n \to \infty}H((HKL)\int_{T}f_n(w)d\mcM(w), (HKL)\int_{T}f(w)d\mcM(w))=0.
\end{eqnarray*}
\end{thm}
\begin{proof}
Given $f,~f_n: T \to \mathbb{R},~n \in \mathbb{N}$ are $(HKL)\mcM-$integrable, then there exist $(HKL)\int_{T}f_n(w)d\mcM(w),~(HKL)\int_{T}f(w)d\mcM(w) \in P_{wkc}(X),~n \in \mathbb{N}$ such that 
\begin{eqnarray*}
\sigma(x^\ast, (HKL)\int_{T}f_n(w)d\mcM(w))=(HK)\int_{T}f_n(w)d \sigma(x^\ast, \mcM(w))
\end{eqnarray*}
and \begin{eqnarray*}
\sigma(x^\ast, (HKL)\int_{T}f(w)d\mcM(w))=(HK)\int_{T}f(w)d \sigma(x^\ast, \mcM(w))
\end{eqnarray*}
for each $x^\ast \in X^\ast$ and $n \in \mathbb{N}.$ The properties of support function gives:
\begin{align*}
& H((HKL)\int_{T}f_n(w)d\mcM(w), (HKL)\int_{T}f(w)d\mcM(w))\\&=\sup\limits_{||x^\ast||\leq 1 }|\sigma(x^\ast, (HKL)\int_{T}f_n(w)d \mcM(w))-\sigma(x^\ast, (HKL)\int_{T}f(w)d\mcM(w))|\\&=\sup\limits_{||x^\ast||\leq 1}|(HK)\int_{T}f_n(w)d\sigma(x^\ast, \mcM(w))- (HK)\int_{T}f(w)d \sigma(x^\ast, \mcM(w)|\\&\leq \sup\limits_{||x^\ast||\leq 1}(HK)\int_{T}|f_n(w)-f(w)|d |\sigma(x^\ast, \mcM(.))|(w)\\&\leq (HK)\int_{T}|f_n(w)-f(w)|d |\mcM|(w).
\end{align*}
Now using Vitali's Convergence Theorem
\begin{eqnarray*}
H((HKL)\int_{T}f_n(w), d \mcM(w), (HKL)\int_{T}f(w)d \mcM(w))\\ \leq (HK)\int_{T}|f_n(w)-f(w)|d |\mcM|(w) ~\to 0~ as ~n \to \infty.
\end{eqnarray*}
\end{proof}
We can construct $(KL)-$type dominated convergence theorem for the set-valued \emph{Kluv\'{a}nek-Lewis-Henstock integral} as follows:
\begin{thm}
Let $\mcM: \Sigma \to P_{wkc}(X)$ be a set-valued measure and let $f,~f_n: T \to \mathbb{R},~ n \in \mathbb{N}$ be $(HKL)\mcM-$integrable such that $\lim\limits_{n \to \infty}f_n=f|\mcM|-$a.e. If there exists a non negative $|\mcM|-$integrable function $g:T \to \mathbb{R}$ such that $|f_n(w)|\leq g(w)$ for all $n \in \mathbb{N}$ and $w \in T$ then 
\begin{eqnarray*}
\lim\limits_{n \to \infty}H((HK)\int_{T}f_n(w)d \mcM(w), (HK)\int_{T}f(w)d \mcM(w))=0.
\end{eqnarray*}
\end{thm}
\begin{proof}
Proof is similar as \cite[Theorem 12]{Cai}.
\end{proof}

\section{Declaration}
{\bf Funding:} Not Applicable, the research is not supported by any funding agency.\\
{\bf Conflict of Interest/Competing interests:} The authors declare that there are no  conflict of interest.\\
{\bf Availability of data and material:} The article does not contain any data for
analysis.\\
{\bf Code Availability:} Not Applicable.\\
{\bf Author's Contributions:} All the authors have equal contribution for the preparation of the article.

\end{document}